\documentclass[10pt]{amsart}
\usepackage {amssymb}
\usepackage {amsmath}
\usepackage {graphicx}
\usepackage {hyperref}
\usepackage{amsthm,amsfonts,amsxtra,amssymb,amscd}
\usepackage[all]{xy}
\xyoption{2cell}
\usepackage{enumerate}
\newtheorem{definition}{Definition}[section]
\newtheorem{lemma}{Lemma}[section]

\newtheorem{theorem}{Theorem}[section]

\newtheorem{remark}{Remark}[section]

\newtheorem{proposition}{Proposition}

\def\A{{\mathcal A}}
\def\Z{{\mathbb Z}}
\def\R{{\mathbb R}}

\def\ome{{\omega}}

\def\N{{\mathbb N}}

\def\l{{\ell}}\def\Gama{{\Gamma_S}}

\newcommand{\T}{\mathbb{T}}
\newcommand{\ii}{{\rm i}}
\newcommand{\Gc}{{\mathcal{G}}}
\newcommand{\sma}{{\sigma}}
\title[Characteristic polynomials   related to the   NLS]
{Characteristic polynomials of color marked graphs, related to the
normal form of the non linear Schr\"{o}dinger
equation}\author{Nguyen Bich Van*}
\thanks{*Universit\`a di Roma,   La Sapienza}
\begin{document}
\maketitle
\section{Introduction}
\subsection{Some back ground}\label{background}
Consider the non linear Schr\"{o}dinger equation(NLS):
\begin{equation}\label{01}
iu_t-\triangle u=\kappa |u|^qu+\partial_{\bar{u}}G(|u|^2), q\geq
1\in \N
\end{equation}
 where $u=u(t,\varphi),\varphi\in\T^n$, $G(a)$ is a real analytic
function whose Taylor series start from the degree $q+2$.  One can
rescale the constant $\kappa=\pm1$.  Passing to the Fourier
representation:
\begin{equation}\label{02}
u(t,\varphi)=\sum_{k\in\Z^n}u_k(t)e^{(k,\varphi)};[u]_k:=u_k.
\end{equation}
It is well-known that the equation \eqref{01} can be written as an
infinite dimensional Hamiltonian dynamical system
$\dot{u}=\{H,u\}$, with Hamiltonian:
\begin{equation}\label{03}
H:=\sum_{k\in\Z^n}|k|^2u_k\bar{u}_k+\sum_{k\in\Z^n:\sum_{i=1}^{2q+2}(-1)^ik_i=0}u_{k_1}\bar{u}_{k_2}u_{k_3}\bar{u}_{k_4}. . . u_{2q+1}\bar{u}_{2q+2}+[G(|u|^2)]_0
\end{equation}
on the scale of complex Hilbert spaces:
\begin{equation}\label{03'}
\bar{\ell}^{a,p}:=\{u=\{u\}_{k\in\Z^n}|\sum_{k\in\Z^n}|u_k|^2e^{2a|k|}|k|^{2p}:=||
u||^2_{a,p}<\infty; a>0, p>n/2\}
\end{equation}
For $\epsilon$ sufficient small there is a analytic change of
variables which brings \eqref{03} to $H=H_N+P^{2q+2}(u)$, where
$P^{2q+2}(u)$ is analytic of degree at least $2(q+2)$ in $u$
while:
\begin{equation}\label{03''}
H_N:=\sum_{k\in\Z^n}|k|^2u_k\bar{u}_k+\sum_{\alpha,\beta\in(\Z^n)^{\N}:|\alpha|=|\beta|=q+1;\sum_k(\alpha_k-\beta_k)k=0,\sum_k(\alpha_k-\beta_k)|k|^2=0}\left(%
\begin{array}{c}
                                                                                                                                                           q+1 \\
                                                                                                                                                           \alpha \\
                                                                                                                                                         \end{array}%
                                                                                                                                                         \right) \left(%
\begin{array}{c}
  q+1 \\
  \beta \\
\end{array}%
\right)u^\alpha \bar{u}^\beta
\end{equation}
 Let us now partition:
\[\Z^n=S\cup S^c; S=\{v_1,. . . ,v_m\}\]
where $S$ is called \emph{tangential sites}, it is some
(arbitrarily large) subset of $\Z^n$ satisfying the
\emph{completeness condition}, $S^c$-\emph{normal sites}.
\newline We set
\begin{equation}\label{04}
u_k:=z_k, k\in S^c, u_{v_i}:= \sqrt {\xi_i+y_i} e^{\ii x_i}= \sqrt
{\xi_i}(1+\frac {y_i}{2 \xi_i }+\ldots  ) e^{\ii x_i}\;{\rm for}\;
i=1,\dots ,m,
\end{equation}
considering $\xi_i$ as parameters, $|y_i|<\xi_i$, while
$y,x,w:=z,\bar{z}$ are dynamical variables.  We separate $H=N+P$
where $N$ is the \emph{normal} form and collects all the terms of
$H_N$ of degree $\leq 2$.  We introduce
\begin{equation}\label{05}
A_r(\xi_1,. . . ,\xi_m)=\sum_{\sum_ik_i=r}\left(%
\begin{array}{c}
  r \\
  k_1,. . . ,k_m \\
\end{array}%
\right)^2\prod_i \xi_i^{k_i}
\end{equation}
By Proposition 4. 4 in \cite{CM2} we have
\begin{equation}\label{06}
N=(\omega(\xi),y)+\sum_k\Omega_k(\xi)|z_k|^2+Q_M(x,\omega)
\end{equation}
where \begin{equation}\label{07} \omega=\omega_0+\nabla_\xi
A_{q+1}(\xi),\Omega_k=|k|^2+(q+1)^2A_{q+1}(\xi)
\end{equation}
and $Q_M$ is given by formula \ref{08}.  Let $\{e_1,. . . ,e_m\}$ be a
basis of $\Z^m$.
\begin{definition}(edges)\label{edge}
Consider the elements: \begin{equation}
\label{edg}X_q:=\{\ell=\sum_{j=1}^{2q}\pm
e_{i_j}=\sum_{i=1}^{m}\ell_ie_i, \ell \neq
0,-2e_i,\eta(\ell)\in\{0,-2\}\}
\end{equation}
The {\em support} of an edge $\ell=\sum_in_ie_i$ is the set of indices $i$ with $n_i\neq 0$.
\end{definition}We have $\sum_i|\ell_i|\leq 2q$ and have imposed the
mass constraint $\sum_i\ell_i=\eta(\ell)\in\{0,-2\}$.  We call all
the elements respectively the \emph{black}, $\eta(\ell)=0$ and
\emph{red} $\eta(\ell)=-2$\emph{ edges} and denote them by $X^0_q,
X_q^{-2}$ respectively.   Notice that by our constraints  the support of an edge contains at least 2 elements.
\begin{definition}
\begin{itemize}
\item When $\ell\in X_q^0$, we define $\mathcal{P}_\ell$ as the
set of pairs $k,h$ satisfying $\sum_{j=1}^m\ell_jv_j+k-h=0;
\sum_{j=1}^m\ell_j|v_j|^2+|k|^2-|h|^2=0$.
\item When $\ell\in X_q^{-2}$, we define $\mathcal{P}_\ell$ as the
set of unordered pairs $\{h,k\}$ satisfying
$\sum_{j=1}^m\ell_jv_j+k+h=0;
\sum_{j=1}^m\ell_j|v_j|^2+|k|^2+|h|^2=0$.
\end{itemize}
\end{definition}
For every edge $\ell$, set $\ell=\ell^+-\ell^-$ and define
\begin{equation}\label{06'}
c(\ell)=c_q(\ell):=\left\{%
\begin{array}{ll}
    (q+1)^2\xi^{\frac{\ell^++\ell^-}{2}}\sum_{\alpha\in \N^m;|\alpha+\l^+|_1=q}\left(%
\begin{array}{c}
  q \\
  \l^++\alpha\\
\end{array}%
\right)\left(%
\begin{array}{c}
  q \\
  \l^-+\alpha\\
\end{array}%
\right)\xi^{\alpha}, & \hbox{$\l\in X_q^0;$} \\
   (q+1)q\xi^{\frac{\ell^++\ell^-}{2}}\sum_{\alpha\in \N^m;|\alpha+\l^+|_1=q-1}\left(%
\begin{array}{c}
  q+1 \\
  \l^-+\alpha\\
\end{array}%
\right)\left(%
\begin{array}{c}
  q-1 \\
  \l^++\alpha\\
\end{array}%
\right)\xi^{\alpha} , & \hbox{$\l\in X_q^{-2}$. } \\
\end{array}%
\right.
\end{equation}
Then in \cite{CM2} it has been proved that:
\begin{equation}\label{08}
Q_M(x,\ome)=\sum_{\l\in X_q^0}c(\l)e^{\l x}\sum_{(h,k)\in
\mathcal{P}_{\l}}z_h\bar{z}_k+\sum_{\l\in
X_q^{-2}}c(\l)\sum_{{h,k}\in \mathcal{P}_{\l}}(e^{\l
x}z_hz_k+e^{-\l x}\bar{z}_h\bar{z}_k)
\end{equation}

This is a very complicated infinite dimensional quadratic
Hamiltonian, one needs to decompose this infinite dimensional
system into infinitely many decoupled finite dimensional systems.
\newline \begin{definition}Momentum is the linear map
$\pi:\Z^m\rightarrow \Z^n,\pi(e_i)=v_i$.
\end{definition}
Define $iM(x)$ as the matrix of $ad(Q_M)$ in the basis
$z_k,\bar{z}_k$, while $iM$ is the matrix of $ad(Q_M)$in the basis
$e^{i\mu x}z_k,e^{-i\mu x}\bar{z}_k, \pi(\mu)+k=0$.  In \cite {CM2}
it was proved that $M$ is block diagonal with 2 blocks(denoted by
$A,\pm$ in correspondence with each connected component $A$ of the
geometric graph(c. f.  \ref{gegr} ).   The control of these blocks is
then needed to prove further non-degeneracy properties of this
Hamiltonian. \newline Set
\[\Z^m_c=\{\mu\in \Z^m|-\pi(\mu)\in S^c\}. \]
\begin{definition}
The graph $\tilde{\Gamma}_S$ has as vertices the variables $z_h,
\bar{z}_k$ and edges corresponding to the non-zero entries of
matrix $M(x)$ in the geometric basis(i. e.  $z_h,\bar{z}_k$).
\newline The graph $\Lambda$ has as vertices $\Z^m_c \times
\Z/(2)$ and edges corresponding to the non-zero entries of $M$ in
the frequency basis (i. e.  $e^{i\mu x}z_k,e^{-i\mu x}\bar{z}_k,
\pi(\mu)+k=0$).
\end{definition}
\begin{definition} Two points $h,k\in S^c$ are connected by a black
edge if $z_h,z_k$ are connected in $\tilde{\Gamma}_S$, while
$h,k\in S^c$ are connected by a red edge if $z_h,\bar{z}_k$ are
connected in $\tilde{\Gamma}_S$.
\end{definition}
Take a connected component $A$ of $\Gamma_S$. Consider block
$M_{A,+}$. Given two elements $a\neq b\in A$.  By formula \ref{08}
the matrix element $M_{a,b}$ is non-zero if and only if they are
joined by an edge $\l$ and then $M_{a,b}=c(\l)$ if $b=e^{i\mu
x}z_k$ or $M_{a,b}=-c(\l)$ if $b=e^{-i\mu x}\bar{z}_k$.
\[M_{A,-}=-\bar{M}_{A,+}\]
\newline In order to describe the matrix $iN_A$ of $ad(N)$ on A,
we have to finally compute the diagonal terms. One contribution
comes from \ref{07} and assumes the value $\nabla_\xi
A_{q+1}(\xi). \mu$ on the element $e^{i\mu x}z_k$.  \newline In
application of the KAM algorithm to our Hamiltonian a main point
is to prove the validity of the second Melnikov condition.  The
problem arises in the study of the second Melnikov equation where
we have to understand when it is that two eigenvalues are equal or
opposite.  The condition  for a polynomial to have distinct roots
is the non--vanishing of the discriminant while   the condition
for two polynomials to have a root in common  is the  vanishing of
the resultant.   In our case these resultants and discriminants are
polynomials in the parameters $\xi_i$ so, in order to make sure
that the singularities are only in measure 0 sets (in our case
even an algebraic hypersurface), it is necessary to show that
these polynomials are formally non--zero.    This is a purely
algebraic problem involving, in each dimension $n$, only finitely
many explicit polynomials  and so it can be checked by a finite
algorithm.  The problem is that, even in dimension 3, the total
number of these polynomials is quite high (in the order of the
hundreds or thousands) so that the algorithm becomes quickly non
practical.  In order to avoid this  we have experimented with a
conjecture which is stronger than the mere non-vanishing of the
desired polynomials.   We expect our polynomials to be irreducible
and separated, in the sense that the connected component of the graph giving rise to the block and its polynomial can be recovered from the associated characteristic polynomial.
\subsection{A geometric graph}\label{gegr}

To the set $S$ we associate the following configuration,  given
two distinct elements $v_i,v_j\in S$  construct the sphere
$S_{i,j}$ having  the two vectors as opposite points of a diameter
and the two hyperplanes, $H_{i,j},\ H_{j,i}$,  passing through
$v_i$  and $v_j$ respectively, and perpendicular to the line
though the two vectors  $v_i,v_j. $
\smallskip

From this configuration of spheres and pairs of parallel
hyperplanes  we deduce a {\emph combinatorial colored graph},
denoted by $\Gamma_S$, with vertices  the points in $\R^n$ and two
types of edges, which we call {\emph black} and {\emph red}.

\begin{itemize}\item A black edge connects  two points $p\in H_{i,j},\ q\in H_{j,i}$, such that the line $p,q$ is orthogonal to the two hyperplanes, or in other words $q=p+v_j-v_i$.

\item A red edge connects  two points $p,q\in S_{i,j} $ which are opposite points of a diameter.
\end{itemize}
{\bf The Problem}\quad The problem consists in the study of the
connected components of this  graph.   Of course  the nature of the
graph depends upon the choice of $S$  but one expects a relatively
simple behavior for $S$  {\emph generic}.  It is immediate by the
definitions  that the points in $S$ are all pairwise connected by
black and red edges  and it is not hard to see  that, for generic
values of $S$,  the set $S$ is itself a connected component which
we call the {\emph special component}.

\subsection{The Cayley  graphs\label{Cg}}
In order to understand the graph $\Gamma_S$ we develop a formal
setting.  Let $G$ be a group and $X=X^{-1}\subset G$ a
subset. \subsubsection{Marked graphs}
\begin{definition}
An $X$--marked graph is an oriented graph $\Gamma$ such that each
oriented edge is marked with an element $x\in X$.
$$\xymatrix{ &a\ar@{->}[r]^{x} &b&  &a\ar@{<-}[r]^{\ x^{-1}} &b&   }$$
We mark the same edge, with opposite orientation, with $x^{-1}$.
Notice that if $x^2=1$  we may drop the orientation of the edge.
\end{definition}
 \subsubsection{Cayley graphs}
 A typical way to construct an $X$--marked graph is the following.  Consider an action $G\times A\to A$ of $G$ on a set $A$, we then define.
 \begin{definition}[Cayley graph] The graph $A_X$ has as  vertices   the elements of $A$ and, given $a,b\in A$ we join them by an oriented edge $a\stackrel{x}\rightarrow b$, marked $x$, if $b=xa,\ x\in X$.
 \end{definition}
A special case is obtained when $G$ acts on itself by left (resp.
right) multiplication and we have the Cayley graph $G_X^l$ (resp.
$G_X^r$).    We  concentrate on $G_X^l$ which we just denote by
$G_X$.
\subsubsection{The linear rules}
  Denote by  ${\Z^m}:=\{\sum_{i=1}^ma_ie_i,\  a_i\in\Z\}$  the lattice  with basis the elements $e_i$.
We consider  the group $G:={\Z^m}\rtimes\Z/(2)$ semi--direct
product.  Its elements are  couples $(a,\sigma)$ with $a\in \Z^m$,
$\sigma=\pm 1$.   It will be notationally convenient to identify by
$a$ the element $(a,+1)$ and by $\tau$ the element $(0,-1)$.  Note
the commutation rules $a\tau= \tau (-a)$.
 Sometimes we refer to the elements $a=(a,+1)$  as {\em black} and $a\tau=(a,-1)$  as {\em red}.

\begin{definition}
We set $\Lambda$ to be the Cayley graph associated to the elements
$X_q:=X^0_q\cup X^{-2}_q$.
\end{definition}
\subsubsection{From the combinatorial to the geometric graph}
In our geometric setting, we have chosen a list $S$ of vectors
$v_i$ and   we then define  $\pi:\Z^m\to \R^n$ by $\pi: e_i\mapsto
v_i$.

We then think of $G$ also as linear operators on $\R^n$ by setting
\begin{equation}\label{azione}
a  k:=  -\pi(a)+ k,\ k\in\R^n,\ a\in \Z^m\,,\quad \tau k= - k
\end{equation} We  extend $\pi:\Z^m\to \R^n$  to $\Z^m\rtimes \Z/(2)$
  by setting $\pi( a\tau):= \pi(a)  $ so that $-\pi$ is just the orbit map of 0 associated to the action \eqref{azione} (the sign convention is suggested by the conservation of momentum in the NLS).

  We then have that $X$ defines also a Cayley graph on $\R^n$  and  in fact the graph  $\Gamma_S$ is a subgraph of this graph.

There are symmetries in the graph.  The symmetric group $S_m$ of
the $m!$ permutations of the elements $e_i$ preserves the graph.
We have the right actions of $G$, on  the graph: The sign change
\begin{equation} \label{sitr}(b,\sigma)\mapsto (b,\sigma)\tau=b
\sigma\tau,\quad (b,\sigma)\mapsto (b,\sigma) a= (b+\sigma
a,\sigma),\ \forall a,b\in{\Z^m}.
\end{equation}
 Up to the $G$ action  any subgraph an be translated to one containing 0.

 We give definitions which are useful to describe the graphs that appear in our construction.
\begin{definition}\label{CMG}
A {\em complete marked graph}, on a set $A\subset
{\Z^m}\rtimes\Z/(2)$ is the full sub--graph generated by the
vertices   in $A$.
\end{definition}

\begin{definition}
A  graph $A$ with $k+1$ vertices is  said to be of {\em dimension}
$k$
\end{definition}

\subsection{Characteristic polynomials of complete color marked graphs}
As we said in \ref{background} for every complete color marked
graph $\mathcal{G}$ we will consider the matrix $C_{\mathcal{G}}$
indexing by vertices of $\mathcal{G}$ as computed in \cite{CM2}  \S 11. 1. 1:\smallskip

Given  $(a,\sigma), a=\sum_{i=1}^{m}n_ie_i$ set
    \begin{equation}\label{1}
   (q+1) a(\xi):=
 \sum_{i=1}^{m}n_i\frac{\partial}{\partial\xi_i}A_{q+1}(\xi)  \end{equation}
then
\begin{itemize}
    \item In the diagonal at the position $(a,\sigma), a=\sum_{i=1}^{m}n_ie_i$ we put
    \begin{equation}\label{1}
    \begin{cases}
(q+1) a(\xi)\quad\text{if}\ \sigma=1\\
 -(q+1) a(\xi)+2(q+1)^2A_q(\xi)\quad\text{if}\ \sigma=-1
\end{cases}\end{equation}
\item At the position $((a,\sigma_a),(b,\sigma_b))$ we put 0 if
they are not
    connected, otherwise we put $\sigma_b c(\ell)$ (c. f.  \ref{06'}, where $\ell$ is
    the edge
    connecting
    $a,b$.
    \end{itemize}
Define
$\chi_{\Gc}=\chi_{C_{\mathcal{G}}}(t)=det(tI-C_\mathcal{G})$- the
characteristic polynomial of $C_\mathcal{G}$.
\newline As we said in \ref{background} in order to check the
second Melnikov condition we expect that $\chi_{\mathcal{G}}$ are
irreducible over $\Z$ and separated.  In \cite{CM3} we have proved
this  for the case $q=1,n\in\N$.  Here we start to prove
irreducibility and separation for bigger $q$ and low dimensions.
\section{Irreducibility of
characteristic polynomials }

    \begin{lemma}\label{lem1}
    For any $a\in\Z^m$: $a(\xi)$ has integer coefficients.
    \end{lemma}
    \begin{proof}Let $a=\sum_in_ie_i$.  We have
    \[\frac{\partial}{\partial\xi_i}A_{q+1}(\xi)=\sum_{\beta\in\mathbb{N}^m;|\beta|_1=q+1;\beta_i\geq
    1}(\begin{array}{c}q+1\\\beta\end{array})^2\beta_i\xi_1^{\beta_1}. . . \xi_i^{\beta_i-1}. . . \xi_m^{\beta_m}\]
\[(\begin{array}{c}q+1\\\beta\end{array})^2\beta_i=(\begin{array}{c}q+1\\\beta\end{array})(\begin{array}{c}q\\\beta_1,. . . ,\beta_i-1,. . . ,\beta_m\end{array})(q+1)\]
is divisible by $q+1$.  \end{proof}  Hence all diagonal elements
of $C_\mathcal{G}$ are divisible by $q+1$. Besides by the formula
\ref{06'} all off-diagonal elements of $C_\mathcal{G}$ are also
divisible by $q+1$.  Thus we can write:
\[C_\mathcal{G}=(q+1)\tilde{C}_\mathcal{G}\Rightarrow \chi_{C_{\mathcal{G}}}(t)=det(tI-C_\mathcal{G})=det((q+1)\tilde{t}I-(q+1)\tilde{C_\mathcal{G}})=(q+1)^{n+1}\chi_{\tilde{C_\mathcal{G}}}(\tilde{t}) \]
So in order to prove the irreducibility and the separation of the
polynomials $\chi_{C_\mathcal{G}}$ it is enough to prove the
irreducibility and the separation of  the polynomials
$\chi_{\tilde{C}_\mathcal{G}}$.  For simplicity we will denote
$\chi_{\tilde{C}_\mathcal{G}}$ also by $\chi_{\mathcal{G}}$, and
we will redefine  $  c(\l)$ by division the right hand sides of
\ref{06'} by $q+1$:
 \begin{equation}\label{06''}
c(\ell)=c_q(\ell):=\left\{
\begin{array}{ll}
    (q+1)\xi^{\frac{\ell^++\ell^-}{2}}\sum_{\alpha\in \N^m;|\alpha+\l^+|_1=q}\left(%
\begin{array}{c}
  q \\
  \l^++\alpha\\
\end{array}
\right)\left(
\begin{array}{c}
  q \\
  \l^-+\alpha\\
\end{array}
\right)\xi^{\alpha}, & \hbox{$\l\in X_q^0;$} \\
   q\xi^{\frac{\ell^++\ell^-}{2}}\sum_{\alpha\in \N^m;|\alpha+\l^+|_1=q-1}\left(%
\begin{array}{c}
  q+1 \\
  \l^-+\alpha\\
\end{array}
\right)\left(
\begin{array}{c}
  q-1 \\
  \l^++\alpha\\
\end{array}
\right)\xi^{\alpha} , & \hbox{$\l\in X_q^{-2}$. } \\
\end{array}
\right.
\end{equation}

Take a complete colored marked graph ${\mathcal A}$  and compute
its characteristic polynomial $\chi_\A(t)$.  We have:
\begin{theorem}\label{lafatt}
 When we  set a variable
$\xi_i=0$ in  $\chi_\A(t)$ we obtain the product  of the polynomials
$\chi_{\A_i}(t)$  where the $A_i$ are the connected components of
the graph obtained from $\A$ by deleting all the edges in which
$i$ appears as index, with the induced markings (with $\xi_i=0$).
\end{theorem}
\begin{proof}
This is immediate from the form of the matrices.
\end{proof}
\begin{remark}\label{rem3. 1}
\begin{equation}
\frac{\partial}{\partial\xi_i}A_{q+1}(\xi)|_{\xi_i=\xi_j}=\frac{\partial}{\partial\xi_j}A_{q+1}(\xi)|_{\xi_i=\xi_j}\forall
i,j
\end{equation}
\end{remark}

\begin{remark}\label{rem1. 2}
Let $b=\sum_{i=1}^{k}n_ie_i,n_i\neq 0;\sum_{i=1}^kn_i=0$. Then:
\begin{equation}\label{eqrem1. 2}
b(\xi)|_{\xi_1=\xi_2=\ldots=\xi_k}=0
\end{equation}
\end{remark}
\begin{proof} By the remark \ref{rem3. 1} we have:
\[b(\xi)|_{\xi_1=\xi_2=. . . =\xi_k}=\sum_{i=1}^{k}n_i\frac{\partial}{\partial \xi_i}A_{q+1}(\xi)|_{\xi_1=\xi_2=. . . =\xi_k}=\frac{\partial}{\partial \xi_1}A_{q+1}(\xi)|_{\xi_1=\xi_2=. . .
=\xi_k}\sum_{i=1}^{k}n_i=0\].\end{proof}
\begin{remark}\label{rem1.4} Let $\l=\l^+-\l^-$ be an edge. We have:

i) If $\l$ is a black edge, then $|\l^+|_1=|\l^-|_1\leq
q$.

ii) If $\l$ is a red edge, then $|\l^+|_1\leq q-1, |\l^-|_1\leq
q+1$.\end{remark}
\begin{proof} By the definition of edges we have : \begin{equation}\label{40} |\l^+|_1+|\l^-|_1\leq 2q.\end{equation} On the other hand:

i) If $\l$ is a black edge, then
\begin{equation}\label{41}|\l^+|_1-|\l^-|_1=0.\end{equation} From \eqref{40} and \eqref{41} we get $|\l^+|_1=|\l^-|_1\leq q.$

ii) If $\l$ is red edge, then
\begin{equation}\label{42} |\l^+|_1-|\l^-|_1=-2.\end{equation}
From \eqref{40} and \eqref{42} we get $|\l^+|\leq q-1,|\l^-|_1\leq
q+1$.
\end{proof}
\begin{remark}\label{rem5}: Let $\l=\sum_{i=1}^kn_ie_i=\l^+-\l^-,n_i\neq 0,$ be an edge.

i) If $\l$ is a black edge and $k=m$, then
$|\l^{+}|_1=|\l^{-}|_1=q$ and $c(\l)=(q+1)\xi^{(\l^++\l^-)/2}\left(%
\begin{array}{c}
  q \\
  \l^+\\
\end{array}%
\right)\left(%
\begin{array}{c}
  q \\
  \l^- \\
\end{array}%
\right)$.

ii) If $\l$ is a red edge and $k=m$, then $|\l^{+}|_1=q-1,
|\l^{-}|_1=q+1$ and  $c(\l)=q\xi^{(\l^++\l^-)/2}\left(%
\begin{array}{c}
  q+1 \\
  \l^-\\
\end{array}%
\right)\left(%
\begin{array}{c}
  q-1 \\
  \l^+\\
\end{array}%
\right)$.\end{remark}
\begin{proof} Since $S=\{v_1,...,v_m\}$ is some arbitrarily large
set, we may suppose $m\geq 2q$. If $k=m$ then
$|\l^+|_1+|\l^-|_1=\sum_{i=1}^mn_i\geq m\geq 2q$. Moreover, by
definition of edges $\sum_{i=1}^mn_i\leq 2q$. Hence:
\begin{equation}\label{38}|\l^+|_1+|\l^-|_1=\sum_{i=1}^mn_i=2q.\end{equation}
i) When $\l$ is a black edge, we have
\begin{equation}\label{37}|\l^+|_1-|\l^-|_1=0\end{equation}
From \eqref{38} and \eqref{37} we get $|\l^+|_1=|\l^-|_1=q$. By
formula \eqref{06''} we obtain $c(\l)=(q+1)\xi^{(\l^++\l^-)/2}\left(%
\begin{array}{c}
  q \\
  \l^+ \\
\end{array}%
\right)\left(%
\begin{array}{c}
  q \\
  \l^- \\
\end{array}%
\right)$.
\newline ii) When $\l$ is a red edge, we have
\begin{equation}\label{39}|\l^+|_1-|\l^-|_1=-2\end{equation} From
\eqref{38} and \eqref{39} we get $|\l^+|_1=q-1,|\l^-|_1=q+1$. By
formula \eqref{06''} we obtain $c(\l)=q\xi^{(\l^++\l^-)/2}\left(%
\begin{array}{c}
  q+1 \\
  \l^-\\
\end{array}%
\right)\left(%
\begin{array}{c}
  q-1 \\
  \l^+\\
\end{array}%
\right)$.
\end{proof}
We finally recall Proposition 14 of \cite{CM2}
\begin{proposition}\label{ilrido} (i) For $n=1$ and for generic choices of $S$, all the connected components of $\Gama$ are either vertices or single edges.

(ii) For $n=2$,  and  for every $m$  there exist infinitely many choices of generic tangential sites $S=\{v_1,\ldots,v_m\}$ such that, if $A$ is a connect component of the geometric graph  $\Gama$, then $A$ is either a vertex or a single  edge.
\end{proposition}
\smallskip

\textbf{Obtained results}:  For graphs reduced to one vertex the
statement is trivial.   At the moment we are able to prove the
irreducibility and separation in  dimension 1,  and dimension 2,
under the assumptions of Proposition \ref{ilrido} for all $q$
since   all graphs which appear have at most one edge.
\subsection{One edge}
\begin{theorem}
For any $q$ and any connected color marked graph with one edge
the characteristic polynomial is irreducible.
\end{theorem}
\begin{proof}  We choose the root so that the graph has one of the forms:
    $$\xymatrix{ &0\ar@{-}[r]^{\l}_{black} &\ell&}\quad \text{
    or}\quad
    \xymatrix{ &0\ar@{=}[r]^{\l}_{red} &\ell&}$$
  Let $ \ell=\sum_{i=1}^kn_ie_i, n_i\neq 0$.
We have
\begin{equation}\label{7}\ell(\xi)=\frac{1}{q+1}\sum_{i=1}^kn_i\frac{\partial}{\partial\xi_i}A_{q+1}(\xi)=\sum_{i=1}^kn_i\sum_{\beta\in\mathbb{N}^m;|\beta|_1=q+1;\beta_i\geq
    1}(\begin{array}{c}q+1\\\beta\end{array})(\begin{array}{c}q\\\beta_1,. . . ,\beta_i-1,. . . ,\beta_m\end{array})\xi_1^{\beta_1}. . . \xi_i^{\beta_i-1}. . . \xi_m^{\beta_m}\end{equation}Set $\bar \ell(\xi):=\ell(\xi)$ if $\eta(\ell)=0$ and $\bar \ell(\xi):=-\ell(\xi)+2(q+1)A_q(\xi)$ if    $\eta(\ell)=-2. $
\begin{remark}\label{rem1}  For every $i$ in the support of $\l$, unless $q=4$ and $\ell=-5e_i+e_j+e_k+e_m$, the polynomial
 $\bar \ell(\xi)$ contains the term $\xi_i^q$ with non zero coefficient.
\end{remark}
 \textbf{Proof:} In the formula of $\ell(\xi)$ there is the
monomial:
$$(n_i+(q+1)\sum_{h\neq i}n_h)\xi_i^q,$$ since $\sum_hn_h=\eta(\ell)$ this equals
$$ -qn_i\xi_i^q\quad\text{if}\ \eta(\ell)=0$$
and $$[n_i+(q+1)(-2-n_i)]\xi_i^q\quad\text{if}\ \eta(\ell)=-2$$
In $A_q(\xi)$ the monomial $\xi_i^q$ appears with coefficient 1,so
we get in $\bar{\l}$ the coefficient of $\xi_i^q$ is:
\begin{equation}\label{coef}
-n_i+(q+1)(2+n_i)+2(q+1)=4(q+1)+qn_i\end{equation} which is   non
zero unless $q=4, n_i=-5$ or $q=2, n_i=-6; q=1, n_i=-4$, by
Formula \eqref{edg} these last two cases do  not occur since
$|n_i|\leq 2q$.  As for $q=4$  an edge is a sum of  8 elements
$e_i$ with sign $\pm 1$, if there is a coefficient $-5$ at least 5
elements have coefficient $-1$, then there must be 3 elements with
coefficient 1 to give $\eta(\ell)=-2$.
This extra case will be considered at the end of this subsection.
$\Box$

We now compute with the matrix
\[C_{\mathcal{G}}=\left(%
\begin{array}{cc}
  0 & \sigma_\l c(\l) \\
    c(\l) &  \bar\l(\xi) \\
\end{array}%
\right)\]
\begin{equation}\label{5}
\chi_{\Gc}(t)=det \left(%
\begin{array}{cc}
  t & -\sigma_\l c(\ell)\\
  - c(\ell) & t-\bar\l(\xi) \\
\end{array}%
\right)=t^2-\bar\l(\xi)t-\sigma_\l c(\l)^2.
\end{equation}
Suppose that $\chi_{\Gc}$ is not irreducible, then:
\begin{equation}\label{6'}
\chi_{\Gc}(t)=(t+r(\xi))(t-\bar\l(\xi) -r(\xi)).
\end{equation}
Compare the free coefficients in \ref{5} and \ref{6'} we get
\begin{equation}\label{6*}
r(\xi)(-\bar\l(\xi) -r(\xi))=-\sigma_\l c(\l)^2.
\end{equation}
By the formula \ref{06'} $c(\l)^2$ is divisible by
$\xi_i^{|n_i|},\forall i=1,. . . ,k$.  \newline For any $i$ if
$r(\xi)$ is divisible by $\xi_i$, by remark \ref{rem1}
$\bar{\l}(\xi)$ is not divisible by $\xi_i$, then $-\bar\l(\xi)
-r(\xi)$ is not divisible by $\xi_i$.  And inversely, if
$-\bar\l(\xi) -r(\xi)$ is divisible by $\xi_i$, then $r(\xi)$ is
not divisible by $\xi_i$. Hence we have:
\begin{eqnarray}\label{sys_eq1}
  r(\xi)&=&\xi_i^{|n_i|}s_i, i\in A  \\
  -\bar\l(\xi)-r(\xi) &=& \xi_j^{|n_j|}u_j, j\in B.
\end{eqnarray}
where $A\cup B=\{1,. . . ,k\}; A\cap B=\emptyset$.  \begin{enumerate}
    \item If $A\neq \emptyset$ and $B\neq\emptyset$, then for some couple
$i,j$ we have:
\begin{equation}\label{11}
\bar\l(\xi)= -(\xi_i^{|n_i|}s_i+\xi_j^{|n_j|}u_j)
\end{equation}
From   remark \ref{rem1} we must
have $n_h=0,\forall h\neq i,j$,
\begin{enumerate}
\item\label{black}
 \textbf{When $\l$ is a black edge:}
\newline We have  $\sma_\l=1$ and by the definition of edge (cf.  \ref{edge}) $\l=ne_i-ne_j;2|n|\leq 2q$.  We may suppose $i=1,j=2,n>0$.  We have $\bar\l(\xi)=\l(\xi)$ and:
\begin{multline}\label{12}
\l(\xi) =n(\sum_{\beta\in
\mathbb{N}^m;|\beta|_1=q+1,\beta_1\geqslant 1}(\begin{array}{c}q+1
\\\beta\\\end{array})
(\begin{array}{c}
                       q \\
                       \beta_1-1,\beta_2,. . . ,\beta_m \\
                     \end{array})
                     \xi_1^{\beta_1-1}\xi_2^{\beta_2}. . . \xi_m^{\beta_m}-\\
                     -\sum_{\beta'\in
\mathbb{N}^m;|\beta'|_1=q+1;\beta'_2\geqslant
1}(\begin{array}{c}q+1
\\\beta'\\\end{array})(\begin{array}{c}
                      q \\
                       \beta'_1,\beta'_2-1,. . . ,\beta'_m \\
                     \end{array})\xi_1^{\beta'_1}\xi_2^{\beta'_2-1}. . . \xi_m^{\beta'_m})
\end{multline}
Remark that
\[\xi_1^{\beta_1-1}\xi_2^{\beta_2}. . . \xi_m^{\beta_m}=\xi_1^{\beta'_1}\xi_2^{\beta'_2-1}. . . \xi_m^{\beta'_m}\Leftrightarrow \beta_1-1=\beta'_1,\beta_2=\beta'_2-1,\beta_i=\beta'_i\forall i\geqslant 3\]
Then:
\begin{multline}\label{13}
\l(\xi) =n\sum_{\beta\in\mathbb{N}^m,|\beta|_1=q+1,\beta_1\geqslant
1}\frac{q!}{(\beta_1-1)!\beta_2!. . . \beta_m!}\frac{(q+1)!}{\beta_1!. . . \beta_m!}(1-\frac{\beta_1}{\beta_2+1})\xi_1^{\beta_1-1}\xi_2^{\beta_2}. . . \xi_m^{\beta_m}
\end{multline}
By \ref{11} we must have
\begin{equation}\label{14}
\l(\xi) =-(\xi_1^ns_1+\xi_2^nu_2).
\end{equation}
\begin{enumerate}
\item If $n>1$, we take $\beta_1=1,\beta_2=n-1,\beta_3=q+1-n,\beta_4=. . . =\beta_m=0$, then in the
formula \ref{13} of $\l(\xi) $, there is the monomial
\[n\frac{q!}{(n-1)!(q+1-n)!}\frac{(q+1)!}{(n-1)!(q+1-n)!}(1-\frac{1}{n})\xi_2^{n-1}\xi_3^{q+1-n}\neq
0\] and they are not divisible by $\xi_1^n$ or $\xi_2^n$.  This
contradicts \ref{14}.
\item $n=1$.  We have $\l^+=(1,0,. . . ,0);\l^-=(0,1,. . . ,0)$. Then from
\ref{06''} we get
\begin{equation}\label{15''}
c(\l)^2=(q+1)^2\xi_1\xi_2(\sum_{\alpha\in\N^m:\sum_ia_i+1=q}\left(%
\begin{array}{c}
  q \\
  \alpha_1+1,\alpha_2,. . . ,\alpha_m \\
\end{array}%
\right)\left(%
\begin{array}{c}
  q \\
  \alpha_1,\alpha_2+1,. . . ,\alpha_m \\
\end{array}%
\right)\xi^\alpha)^2
\end{equation}
\newline Let $p$ be a prime divisor of $q+1$:$q+1=p^ku,g. c. d(p,u)=1$.  We have:
\begin{equation}\label{15}
\chi_\Gc=t(t-\l(\xi) )(\mbox{mod } p)\Rightarrow
\chi_\Gc=(t+ps)(t-ps-\l(\xi) )
\end{equation}
By \ref{5}and \ref{15''} the free coefficient of $\chi_\Gc$ must
be divisible by $p^{2k}$:
\begin{equation}\label{15'}p^{2k}|ps(-\l(\xi) -ps)\end{equation}
By formula \eqref{13} we see that the coefficient of the term
$\xi_1^q$ is $-q$, the coefficient of the term $\xi_2^q$ is $q$.
One deduces that $\l(\xi) $ is not divisible by $p$ since $g. c.
d(q,q+1)=1$. Hence $(-\l(\xi) -ps)$ is not divisible by $p$. So by
\ref{15'} we must have $p^{2k-1}|s$.  Now take
$\xi_1=\xi_2\Rightarrow \l(\xi) =0$, then the free coefficient of
$\chi_\Gc$ when $\xi_1=\xi_2$ is divisible by $p^{4k}$.  But in
\ref{5} when $\xi_1=\xi_2$ the free coefficient of $\chi_\Gc$ is
$-c(\l)^2|_{\xi_1=\xi_2}$, it is not divisible by $p^{4k}$, since
in \ref{15''} if we take $\alpha_1=\alpha_2=0,\alpha_3=q-1$, we
have the monomial:
\[(q+1)^2\xi_1^2(q^2\xi_3^{q-1})^2\]
is not divisible by $p^{4k}$.
\end{enumerate}
\item When $\l$ is a red edge:
When $h\neq i,j,n_h=0$ we get the coefficient of the term
$\xi_h^q$ in $\bar{\l}(\xi)$ is $4(q+1)+qn_h=4(q+1)\neq 0$, so
\ref{11} cannot hold.
\end{enumerate}
\item If $B=\emptyset$, then
$A=\{1,. . . ,k\}$\begin{equation}\label{_r}r(\xi)=\xi_1^{|n_1|}.
. . \xi_k^{|n_k|}s\end{equation}.\begin{enumerate}\item{When $\l$
is a black edge:} Take $\xi_1=. . . =\xi_k$, by the remark
\label{rem4} we have $\l(\xi) |_{\xi_1=. . . \xi_k}=0$, hence
\begin{multline}\label{17}
\chi_\Gc(t)|_{\xi_1=. . . =\xi_k}=(t+r(\xi)|_{\xi_1=. . . =\xi_k})(t-r(\xi)|_{\xi_1=. . . =\xi_k})=\\=t^2-r(\xi)|_{\xi_1=. . . =\xi_k}^2.
\end{multline}
By \ref{17} the free coefficient of $\chi_\Gc|_{\xi_1=. . .
=\xi_k}$ is divisible by $\xi_1^{2\sum_{i=1}^k|n_i|}$.  But by
\ref{5} the free coefficient of $\chi_G|_{\xi_1=. . . =\xi_k}$ is
$-c(\l)^2|_{\xi_1=...=\xi_k}$.

-If $k=m$, then by remark \ref{rem5} $-c(\ell)^2|_{\xi_1=. . .
=\xi_k}=-(q+1)^2\xi_1^{\sum_{i=1}^k|n_i|}\left(%
\begin{array}{c}
  q \\
  \l^+ \\
\end{array}%
\right)^2\left(%
\begin{array}{c}
  q \\
  \l^- \\
\end{array}%
\right)^2$ is not divisible by $\xi_1^{2\sum_{i=1}^k|n_i|}$.

-If $k<m$, then
\begin{equation}
-c(\ell)^2|_{\xi_1=. . . =\xi_k}=-(q+1)^2\xi_1^{\sum_{i=1}^k|n_i|}(\sum_{\alpha\in\mathbb{N}^m:|\ell^++\alpha|_1=q}(\begin{array}{c}
                                     q \\
                                                                                                                \ell^++\alpha \\
                                                                                                              \end{array})(\begin{array}{c}
                                     q \\
                                                                                                                \ell^-+\alpha \\
                                                                                                              \end{array})\xi_1^{\sum_{i=1}^k\alpha_i}\xi_{k+1}^{\alpha_{k+1}}. . . \xi_m^{\alpha_m})^2
\end{equation}
Take $\alpha_1=...=\alpha_k=0,\alpha_{k+1}=q-|\l^+|_1$, we see
that $-c(\l)^2|_{\xi_1=...\xi_k}$ contains the term
$\xi_1^{\sum_{i=1}^k|n_i|}\xi_{k+1}^{2(q-|\l^+|_1)}$ with the
coefficient $-(q+1)^2(\begin{array}{c}
                                     q \\
                                                                                                                \ell^++\alpha \\
                                                                                                              \end{array})^2(\begin{array}{c}
                                     q \\
                                                                                                                \ell^-+\alpha \\
                                                                                                              \end{array})^2$. Hence $-c(\l)^2|_{\xi_1=...=\xi_k}$ is not divisible by
$\xi_1^{2\sum_{i=1}^k|n_i|}$.
\item{When $\l$ is a red edge:}
Take $\xi_1=...=\xi_k$, we have
\begin{multline}\label{09}
\frac{\partial}{\partial
\xi_i}A_{q+1}(\xi)=\frac{\partial}{\partial
\xi_j}A_{q+1}(\xi)\forall i,j\Rightarrow
\l(\xi)|_{\xi_1=...=\xi_k}=\sum_{i=1}^{k}n_i\frac{\partial}{\partial
\xi_1}A_{q+1}(\xi)=\\=-2\frac{\partial}{\partial
\xi_1}A_{q+1}(\xi)=-2\sum_{|\alpha|_1=q+1,\alpha_1\geq
1}\frac{1}{q+1}\left(%
\begin{array}{c}
  q+1 \\
  \alpha \\
\end{array}%
\right)^2
\alpha_1\xi_1^{\alpha_1+\alpha_2+...+\alpha_k-1}\xi_{k+1}^{\alpha_{k+1}}...\xi_{m}^{\alpha_m}
\end{multline}
\begin{equation}\label{10}A_q(\xi)|_{\xi_1=...=\xi_k}=\sum_{\beta:|\beta|_1=q}\left(%
\begin{array}{c}
  q \\
  \beta \\
\end{array}%
\right)^2\xi_1^{\beta_1+...+\beta_k}\xi_{k+1}^{\beta_{k+1}}...\xi_m^{\beta_m}.\end{equation}
From \eqref{09} and \eqref{10} we have
\begin{multline}\label{36}-\bar{\l}(\xi)|_{\xi_1=...=\xi_k}=(\l(\xi)-2(q+1)A_q(\xi))|_{\xi_1=...=\xi_k}=\\=-2\sum_{\alpha:|\alpha|_1=q+1;\alpha_1\geq 1}(\frac{\alpha_1}{q+1}\left(%
\begin{array}{c}
  q+1 \\
  \alpha \\
\end{array}%
\right)^2+(q+1)\left(%
\begin{array}{c}
  q \\
  \alpha_1,...,\alpha_m \\
\end{array}%
\right)^2)\xi_1^{\alpha_1+...+\alpha_k-1}\xi_{k+1}^{\alpha_{k+1}}...\xi_m^{\alpha_m}=\\=-2\sum_{\alpha:|\alpha|_1=q+1;\alpha_1\geq
1}(\frac{\alpha_1}{q+1}(\frac{(q+1)!}{\alpha_1!...\alpha_m!})^2-(q+1)(\frac{q!}{(\alpha_1-1)!...\alpha_m!})^2)\xi_1^{\alpha_1+...+\alpha_k-1}\xi_{k+1}^{\alpha_{k+1}}...\xi_m^{\alpha_m}=\\=2\sum_{\alpha:|\alpha|_1=q+1;\alpha_1>
1}\frac{q!}{(\alpha
_1-1)!...\alpha_m!}\frac{(q+1)!}{\alpha_1!...\alpha_m!}(\alpha_1-1)\xi_1^{\alpha_1+...+\alpha_k-1}\xi_{k+1}^{\alpha_{k+1}}...\xi_m^{\alpha_m}.\end{multline}Hence
$-\bar{\l(\xi)}|_{\xi_1=...=\xi_k}$ is divisible by $\xi_1$. By
\eqref{_r} $r(\xi)|_{\xi_1=...=\xi_k}=\xi_1^{n_1+...+n_k}s$ is
divisible. Then $(-\bar{\l(\xi)}-r(\xi))|_{\xi_1=...=\xi_k}$ is
divisible by $\xi_1$.By \eqref{6*} and \eqref{_r}we
have:\begin{multline}\xi_1^{|n_1|}...\xi_k^{|n_k|}s(-\bar{\l}(\xi)-r(\xi))=c(\l)^2=\xi_1^{|n_1|}...\xi_k^{|n_k|}(\sum_{\alpha\in
\N^m:|\l^++\alpha|_1=q-1}\left(%
\begin{array}{c}
  q-1 \\
  \l^++\alpha\\
\end{array}%
\right)\left(%
\begin{array}{c}
  q+1 \\
  \l^{-}+\alpha \\
\end{array}%
\right)\xi^\alpha)^2\\\implies
s(-\bar{\l}(\xi)-r(\xi))=(\sum_{\alpha\in
\N^m:|\l^++\alpha|_1=q-1}\left(%
\begin{array}{c}
  q-1 \\
  \l^++\alpha\\
\end{array}%
\right)\left(%
\begin{array}{c}
  q+1 \\
  \l^{-}+\alpha \\
\end{array}%
\right)\xi^\alpha)^2\end{multline}\begin{equation}\label{36'}\implies
s(-\bar{\l}(\xi)-r(\xi))=(\sum_{\alpha\in
\N^m:|\l^++\alpha|_1=q-1}\left(%
\begin{array}{c}
  q-1 \\
  \l^++\alpha\\
\end{array}%
\right)\left(%
\begin{array}{c}
  q+1 \\
  \l^{-}+\alpha \\
\end{array}%
\right)\xi^\alpha)^2.\end{equation} So the right hand side of
\eqref{36'} when $\xi_1=\xi_2=...=\xi_k$ must be divisible by
$\xi_1$. But in fact:

- If $k=m$, then by remark \ref{rem5} $$(\sum_{\alpha\in
\N^m:|\l^++\alpha|_1=q-1}\left(%
\begin{array}{c}
  q-1 \\
  \l^++\alpha\\
\end{array}%
\right)\left(%
\begin{array}{c}
  q+1 \\
  \l^{-}+\alpha \\
\end{array}%
\right)\xi^\alpha)^2=\left(%
\begin{array}{c}
  q-1 \\
  \l^+ \\
\end{array}%
\right)^2\left(%
\begin{array}{c}
  q+1 \\
  \l^- \\
\end{array}%
\right)^2$$ is a constant, not divisible by $\xi_1$.

- If $k<m$, take $\tilde{\alpha} $ such that
$\tilde{\alpha}_1=...=\tilde{\alpha}_k=0,\tilde{\alpha}_{k+1}=q-1-\l^+$
then the right hand side of \eqref{36'} contains the monomial
$$\left(%
\begin{array}{c}
  q-1 \\
  \l^++\tilde{\alpha} \\
\end{array}%
\right)^2\left(%
\begin{array}{c}
  q+1 \\
  \l^-+\tilde{\alpha} \\
\end{array}%
\right)^2\xi_{k+1}^{2(q-1-|\l^+|_1)}.$$ Hence the right hand side
of \eqref{36'} is not divisible by $\xi_1$.
\end{enumerate}
\item The case $A=\emptyset,B=\{1,. . . ,k\}$ is similar.
\end{enumerate}
\textbf{Extra case:} \[q=4,\l=-5e_i+e_j+e_k+e_m\]In this case
$\sigma_\l=-1$. By \ref{06''} we have:
\begin{equation}
c(\l)=4\xi_i^{5/2}\xi_j^{1/2}\xi_k^{1/2}\xi_m^{1/2}.
\end{equation}
By \ref{6*}: \begin{equation}\label{20}r(\xi)(-\bar\l(\xi)
-r(\xi))=c(\l)^2=4\xi_i^{5/2}\xi_j^{1/2}\xi_k^{1/2}\xi_m^{1/2}.\end{equation}
Hence $r(\xi),-\bar{\l}(\xi)-r(\xi)$ must be monomials that
contains only variables from $\xi_i,\xi_j,\xi_k,\xi_m$.So
$\bar{\l}(\xi)$ must be a polynomial that contains only variables
$\xi_i,\xi_j,\xi_k,\xi_m$. But in fact, for $h$ that is not in the
support of $\l$,$n_h=0$ and by \eqref{coef} the coefficient of
$\xi_h^q$ in $\bar{l}(\xi)$ is $4(q+1)+qn_h=20$. Hence we have a
contradiction.
\end{proof}

\end{document}